\documentclass{article}%
\usepackage{amssymb}
\usepackage{amsmath}
\usepackage{amsfonts}
\usepackage{graphicx}%
\providecommand{\U}[1]{\protect\rule{.1in}{.1in}}
\newtheorem{theorem}{Theorem}

\newtheorem{corollary}[theorem]{Corollary}

\newtheorem{remark}[theorem]{Remark}

\newenvironment{proof}[1][Proof]{\noindent\textbf{#1.} }{\ \rule{0.5em}{0.5em}}
\begin{document}

\title{Characterization of Lipschitz continuous DC functions\thanks{The research of
the first author has been supported by the CONICYT of Chile (Fondecyt No
1110019 and ECOS-Conicyt No C10E08) and by the MICINN of Spain (grant
MTM2008-06695-C03-02). The research of the second author has been supported by
MICINN of Spain, grant MTM2008-06695-C03-03, by Generalitat de Catalunya and
by the Barcelona GSE Research Network. He is affiliated to MOVE (Markets,
Organizations and Votes in Economics).}}
\author{A.\ Hantoute$^{1}$\thanks{ahantoute@dim.uchile.cl}, J.E.
Mart\'{\i}nez-Legaz$^{2}$\thanks{juanenrique.martinez.legaz@uab.cat}\\$^{1}${\small Universidad de Chile, Centro de Modelamiento Matem\'{a}tico
(CMM)}\\{\small Avda Blanco Encalada 2120, Piso 7, Santiago, Chile}\\{\small \ }$^{2}${\small Universitat Aut\`{o}noma de Barcelona}\\{\small Departament d'Economia i d'Hist\`{o}ria Econ\`{o}mica, 08193
Bellaterra, Spain}}
\date{}
\maketitle

\begin{abstract}
We give a necessary and sufficient condition for a difference of convex (DC,
for short) functions, defined on a locally convex space, to be Lipschitz
continuous. Our criterion relies on the intersections of the $\varepsilon
$-subdifferentials of the involved functions.



\end{abstract}

\section{Introduction}

In this paper, we work with a (Hausdorff) real locally convex topological
vector space $X$ whose dual is denoted by $X^{\ast}.\ $The duality product is
denoted by $\langle\cdot,\cdot\rangle:X\times X^{\ast}\longrightarrow
\mathbb{R},$ and the zero vector (in $X$ and $X^{\ast}$) by $\theta.$

Classical integration formulas (\cite{Moreaubook, Rock70}) have been first
established in the Banach spaces setting for proper lower semicontinuous (lsc,
for short) convex functions using the Fenchel subdifferential, which is
defined for a given function $f:X\rightarrow\mathbb{R}\cup\{+\infty\}$ and a
point $x$ in the domain of $f,\ \operatorname*{dom}f:=\{x\in X\mid
f(x)<+\infty\},$ by
\[
\partial f(x):=\{x^{\ast}\in X^{\ast}:f(y)-f(x)\geq\langle y-x,x^{\ast}%
\rangle\text{ \ \ for all }y\in X\}.
\]
These results have been extended outside the Banach space (\cite{Bachir2002,
ThibMarcelin2005}) and the non-convex settings (\cite{CorYbHan}) by using the
$\varepsilon$-subdifferential mapping, defined for $\varepsilon>0$ by%
\[
\partial_{\varepsilon}f(x):=\{x^{\ast}\in X^{\ast}\mid f(y)-f(x)\geq\langle
y-x,x^{\ast}\rangle-\varepsilon\text{ for all }y\in X\}.
\]

In this paper we exploit an idea, recently used in \cite{K10}, to establish
several characterizations for the Lipschitz character of the difference of
convex (DC, for short) functions. As a consequence, if the Lipschitz constant
is equal to $0$ then we obtain an integration formula guaranteeing the
coincidence of the involved functions up to an additive constant. The main
result is presented in Theorem \ref{dc Lipschitz} in a slightly more general
form, valid in the locally convex spaces setting, which characterizes the
domination of the variations of DC functions by means of a convex continuous
functions. The desired integration formula is obtained in Theorem
\ref{dc Lipschitz copy}.

\section{The main result}

The desired results providing the characterization of Lipschitz DC functions
will be given in Theorem \ref{dc Lipschitz copy}, which is a consequence of
the following theorem.

In what follows, $f,g:X\longrightarrow\mathbb{R\cup}\left\{  +\infty\right\}
$ are two given functions with a common domain%
\[
D:=f^{-1}\left(  \mathbb{R}\right)  =g^{-1}\left(  \mathbb{R}\right)  ,
\]
assumed nonempty and convex.

\begin{theorem}
\label{dc Lipschitz}Let $h:X\longrightarrow\mathbb{R}$ be a continuous convex
function such that $h(\theta)=0.$ Then, the following statements are
equivalent\emph{:}

\emph{(i)} $f$ and $g$ are convex, lsc on $D,$ and satisfy
\[
f(x)-g(x)\leq f(y)-g(y)+h(x-y)\text{ \ for all }x,y\in D.
\]

\emph{(ii) }For each $x\in D$%
\[
\emptyset\neq\partial_{\varepsilon}f(x)\subset\partial_{\varepsilon
}g(x)+\partial_{\varepsilon}h(\theta)\text{ \ \ for all }\varepsilon>0.
\]

\emph{(iii) }For each $x\in D$ there exists $\delta>0$ such that%
\[
\emptyset\neq\partial_{\varepsilon}f(x)\subset\partial_{\varepsilon
}g(x)+\partial_{\varepsilon}h(\theta)\text{ \ \ for all }\varepsilon
\in(0,\delta).
\]

\emph{(iv) }For each $x\in D$%
\[
\partial_{\varepsilon}f\left(  x\right)  \cap\left(  \partial_{\varepsilon
}g\left(  x\right)  +\partial_{\varepsilon}h\left(  \theta\right)  \right)
\neq\emptyset\text{ \ \ for all }\varepsilon>0.
\]

\emph{(v) }For each $x\in D$ there exists $\delta>0$ such that%
\[
\partial_{\varepsilon}f\left(  x\right)  \cap\left(  \partial_{\varepsilon
}g\left(  x\right)  +\partial_{\varepsilon}h\left(  \theta\right)  \right)
\neq\emptyset\text{ \ \ for all }\varepsilon\in(0,\delta).
\]

\end{theorem}

\begin{proof}
(i) $\Longrightarrow$ (ii). Since $f$ is proper ($\operatorname*{dom}%
f\neq\emptyset$), convex and lsc on $D$, for any given $\varepsilon>0$ the
$\varepsilon$-subdifferential operator $\partial_{\varepsilon}f$ is nonempty
on $D$ (\cite[Prop. 2.4.4(iii)]{ZalinescuBook}). For $x\in D$, we define the
function $\widetilde{g}:X\longrightarrow\mathbb{R\cup}\left\{  +\infty
\right\}  $ as%
\[
\widetilde{g}:=g+f\left(  x\right)  -g\left(  x\right)
\]
so that by (i) the inequality $f\leq\widetilde{g}+h(\cdot-x)$ holds, as well
as $f(x)=\widetilde{g}(x)+h(\theta)=\widetilde{g}(x).$ Notice that
$\operatorname*{cl}\widetilde{g}=\operatorname*{cl}g+f(x)-g(x),$ where
$\operatorname*{cl}$ refers to the corresponding lsc envelope\emph{. }Hence,
as $g$ is lsc on $D,$ $\operatorname*{cl}\widetilde{g}$ coincides with
$g+f\left(  x\right)  -g\left(  x\right)  $ on $D,$ which implies that it is
proper. Therefore, since (\cite[Lemma 15]{HanLopZal08})
\[
\operatorname*{cl}(\widetilde{g}+h(\cdot-x))=\operatorname*{cl}\widetilde
{g}+h(\cdot-x)=\operatorname*{cl}g+h(\cdot-x)+f\left(  x\right)  -g\left(
x\right)
\]
and $\partial_{\delta}(\operatorname*{cl}\widetilde{g})(x)=\partial_{\delta
}\widetilde{g}(x)=\partial_{\delta}g(x)$ (for all $\delta>0$), by appealing to
the sum rule of the $\varepsilon$-subdifferential (e.g., \cite[Theorem
2.8.3]{ZalinescuBook}) we get\
\begin{align*}
\partial_{\varepsilon}f\left(  x\right)   &  \subset%
{\displaystyle\bigcup\limits_{\substack{\varepsilon_{1},\varepsilon_{2}%
\geq0\\\varepsilon_{1}+\varepsilon_{2}=\varepsilon}}}
\left(  \partial_{\varepsilon_{1}}(\operatorname*{cl}\widetilde{g})\left(
x\right)  +\partial_{\varepsilon_{2}}h\left(  \theta\right)  \right)  \\
&  =%
{\displaystyle\bigcup\limits_{\substack{\varepsilon_{1},\varepsilon_{2}%
\geq0\\\varepsilon_{1}+\varepsilon_{2}=\varepsilon}}}
\left(  \partial_{\varepsilon_{1}}g\left(  x\right)  +\partial_{\varepsilon
_{2}}h\left(  \theta\right)  \right)  \subset\partial_{\varepsilon}g\left(
x\right)  +\partial_{\varepsilon}h\left(  \theta\right)  ;
\end{align*}
showing that (ii) holds.

The implication (ii) $\Longrightarrow$ (iii) $\Longrightarrow$ (v) and (ii)
$\Longrightarrow$ (iv) $\Longrightarrow$ (v) are obvious.

(v) $\Longrightarrow$ (i). We fix\ $x,y\in D$ and take an arbitrary number
$\varepsilon>0.$ For $m=1,2,\cdots\ $\ we denote
\[
x_{m,i}:=x+\frac{i}{m}\left(  y-x\right)  \text{ \ for }i=0,1,\cdots,m.
\]
Then, by the current assumption (v) for each $i$ and $m$ there exists
$\gamma_{m,i}\in(0,m^{-1})$\ such that
\[
\partial_{m^{-1}\gamma\varepsilon}f\left(  x_{m,i}\right)  \cap\left[
\partial_{m^{-1}\gamma\varepsilon}g\left(  x_{m,i}\right)  +\partial
_{m^{-1}\gamma\varepsilon}h(\theta)\right]  \neq\emptyset\text{ \ for all
}\gamma\in(0,\gamma_{m,i}).
\]
Set%
\[
\gamma_{m}:=\min_{i\in\{1,\cdots,m\}}\gamma_{m,i},
\]
so that $\gamma_{m}>0,$ and choose $u_{m,i}^{\ast}\in\partial_{m^{-1}%
\gamma_{_{m}}\varepsilon}f\left(  x_{m,i}\right)  $, $v_{m,i}^{\ast}%
\in\partial_{m^{-1}\gamma_{_{m}}\varepsilon}g\left(  x_{m,i}\right)  $ and
$w_{m,i}^{\ast}\in\partial_{m^{-1}\gamma\varepsilon}h(\theta)$ such that
$u_{m,i}^{\ast}=v_{m,i}^{\ast}+w_{m,i}^{\ast}$ for $i=1,...,m-1$. In this way,
if $u^{\ast}\in\partial_{\varepsilon}f\left(  x\right)  $ and $v^{\ast}%
\in\partial_{\varepsilon}g\left(  y\right)  $ are given we write
\begin{align*}
f\left(  x_{m,1}\right)  -f\left(  x\right)   &  \geq\frac{1}{m}\left\langle
y-x,u^{\ast}\right\rangle -\varepsilon\\
f\left(  x_{m,i+1}\right)  -f\left(  x_{m,i}\right)   &  \geq\frac{1}%
{m}\left\langle y-x,u_{m,i}^{\ast}\right\rangle -m^{-1}\gamma_{_{m}%
}\varepsilon\text{\qquad}(i=1,...,m-1)\\
g\left(  x_{m,i-1}\right)  -g\left(  x_{m,i}\right)   &  \geq-\frac{1}%
{m}\left\langle y-x,v_{m,i}^{\ast}\right\rangle -m^{-1}\gamma_{_{m}%
}\varepsilon\text{\qquad}(i=1,...,m-1)\\
g\left(  x_{m,m-1}\right)  -g\left(  y\right)   &  \geq-\frac{1}%
{m}\left\langle y-x,v^{\ast}\right\rangle -\varepsilon.
\end{align*}
Adding up these inequalities and using the facts that $x_{m,m}=y$ and
$x_{m,0}=x,$ together with $u_{m,i}^{\ast}=v_{m,i}^{\ast}+w_{m,i}^{\ast}$, we
obtain that
\begin{align*}
f\left(  y\right)  -f\left(  x\right)  +g\left(  x\right)  -g\left(  y\right)
&  \geq\frac{1}{m}\left\langle y-x,u^{\ast}-v^{\ast}\right\rangle +\frac{1}{m}%
{\displaystyle\sum\limits_{i=1}^{m-1}}
\left\langle y-x,w_{m,i}^{\ast}\right\rangle \\
&  \quad\quad\quad-2\left(  m-1\right)  m^{-1}\gamma_{m}\varepsilon
-2\varepsilon.
\end{align*}
Thus, since $w_{m,i}^{\ast}\in\partial_{m^{-1}\gamma\varepsilon}h(\theta)$ we
deduce that
\begin{align*}
f\left(  y\right)  -f\left(  x\right)  +g\left(  x\right)  -g\left(  y\right)
&  \geq\frac{1}{m}\left\langle y-x,u^{\ast}-v^{\ast}\right\rangle -\frac
{m-1}{m}h(x-y)\\
&  \quad\quad\quad-2\left(  m-1\right)  m^{-1}\gamma_{m}\varepsilon
-2\varepsilon
\end{align*}
which gives us, as $m$ goes to $\infty$ (recall that $0<\gamma_{m}\leq m^{-1}%
$),%
\[
f\left(  y\right)  -f\left(  x\right)  +g\left(  x\right)  -g\left(  y\right)
\geq-h(x-y)-2\varepsilon.
\]
Hence, by letting $\varepsilon$ go to $0$ we get
\[
f\left(  x\right)  -g\left(  x\right)  \leq f\left(  y\right)  -g\left(
y\right)  +h(x-y);
\]
that is, (i) follows.\qquad
\end{proof}

\bigskip

The particular case $h:=0$ in Theorem \ref{dc Lipschitz} yields a new
integration result, which relies on the intersection of the $\varepsilon
$-subdifferentials of the nominal functions. We will denote by $f_{D}$ and
$g_{D}$ the restrictions of $f$ and $g$ to $D,$ respectively.

\begin{corollary}
\label{dc constant}\emph{(cf. \cite[Corollary 2.5]{BMR10})} The following
statements are equivalent\emph{:}

\emph{(i)} $f$ and $g$ are convex, lsc on $D,$ and $f_{D}-g_{D}$ is constant.

\emph{(ii) }For each $x\in D$%
\[
\emptyset\neq\partial_{\varepsilon}f(x)\subset\partial_{\varepsilon}g(x)\text{
\ \ for all }\varepsilon>0.
\]

\emph{(iii) }For each $x\in D$ there exists $\delta>0$ such that
\[
\emptyset\neq\partial_{\varepsilon}f(x)\subset\partial_{\varepsilon}g(x)\text{
\ \ for all }\varepsilon\in(0,\delta).
\]

\emph{(iv) }For each $x\in D$%
\[
\partial_{\varepsilon}f\left(  x\right)  \cap\partial_{\varepsilon}g\left(
x\right)  \neq\emptyset\text{ \ \ for all }\varepsilon>0.
\]

\emph{(v) }For each $x\in D$ there exists $\delta>0$ such that%
\[
\partial_{\varepsilon}f\left(  x\right)  \cap\partial_{\varepsilon}g\left(
x\right)  \neq\emptyset\text{ \ \ for all }\varepsilon\in(0,\delta).
\]

\end{corollary}

The following corollary, giving a criterion for integrating the Fenchel
subdifferential, is an immediate consequence of Corollary \ref{dc constant} in
view of the straightforward relationships $\partial f\left(  x\right)
\subset\partial_{\varepsilon}f\left(  x\right)  $ and $\partial g\left(
x\right)  \subset\partial_{\varepsilon}g\left(  x\right)  $ for every $x\in D$
and every $\varepsilon>0.$

\begin{corollary}
\label{subdiff}\emph{(cf. \cite[Theorem 1]{K10})} The following statements are
equivalent\emph{:}

\emph{(i)} For each $x\in D$%
\[
\emptyset\neq\partial f\left(  x\right)  \subset\partial g\left(  x\right)  .
\]
\ \ 

\emph{(ii)} For each $x\in D$%
\[
\partial f\left(  x\right)  \cap\partial g\left(  x\right)  \neq\emptyset.
\]
\ \ 

\emph{(iii)} For each $x\in D$%
\[
\emptyset\neq\partial f\left(  x\right)  =\partial g\left(  x\right)  .
\]
If these statements hold, then $f$ and $g$ are convex, lsc on $D,$ and
$f_{D}-g_{D}$ is constant.
\end{corollary}

\begin{remark}
\label{rem dc constant}\emph{a) The preceding results remain true if }$X$
\emph{is an arbitrary locally convex real topological vector space, not
necessarily Hausdorff. Indeed, the equivalence between the convex and lsc
character of a function and the nonemptiness of its }$\varepsilon
$\emph{-subdifferentials is a reformulation of the Fenchel-Moreau Theorem, the
validity of which in non-Hausdorff spaces has been proved by S. Simons
\cite[Theorem 10.1]{Si11}.}

\emph{b) The equivalence between (i) and (ii) in Corollary \ref{dc constant}
also follows from a well-known characterization of global minima of DC
functions due to J.-B. Hiriart-Urruty \cite[Theorem 4.4]{H89}. Indeed,
according to this characterization, if }$f$\emph{\ and }$g$\emph{\ are convex
then one has }$\partial_{\varepsilon}f\left(  x\right)  \subset\partial
_{\varepsilon}g\left(  x\right)  $\emph{\ for all }$\varepsilon>0$\emph{\ if
and only if }$x$\emph{\ is a global minimum of }$f_{D}-g_{D}.$\emph{\ Hence,
that condition holds for every }$x\in D$\emph{\ if and only if every }$x\in
D$\emph{\ is a global minimum of }$f_{D}-g_{D},$\emph{\ which is obviously
equivalent to }$f_{D}-g_{D}$\emph{\ being constant on }$D$\emph{.}
\end{remark}

From now on we suppose that $X$ is a normed space with a norm denoted by
$\left\Vert \cdot\right\Vert $ whose the dual norm is $\left\Vert
\cdot\right\Vert _{\ast}.$ We use $B_{\ast}\left(  \theta,K\right)  $ to
denote the closed ball in $(X^{\ast},\left\Vert \cdot\right\Vert _{\ast})$
with center $\theta$ and radius $K\geq0,$ and for $A,B\subset X^{\ast}$ we
set
\[
d\left(  A,B\right)  :=\inf\left\{  \left\Vert a-b\right\Vert _{\ast}:a\in
A,\text{ }b\in B\right\}  ,
\]
with the convention that $d\left(  A,B\right)  :=+\infty$ if $A$ or $B$ is empty.

\bigskip

At this moment, we easily get the main result of the paper by taking
$h:=K\left\Vert \cdot\right\Vert $ in Theorem \ref{dc Lipschitz}:

\begin{theorem}
\label{dc Lipschitz copy} Let $K\geq0.$ Then, the following statements are
equivalent\emph{:}

\emph{(i)} $f$ and $g$ are convex, lsc on $D,$ and $f_{D}-g_{D}$ is Lipschitz
with constant $K.$

\emph{(ii) }For each $x\in D$%
\[
\emptyset\neq\partial_{\varepsilon}f(x)\subset\partial_{\varepsilon
}g(x)+B_{\ast}(\theta,K)\text{ \ \ for all }\varepsilon>0.
\]

\emph{(iii) }For each $x\in D$ there exists $\delta>0$ such that
\[
\emptyset\neq\partial_{\varepsilon}f(x)\subset\partial_{\varepsilon
}g(x)+B_{\ast}(\theta,K)\text{ \ \ for all }\varepsilon\in(0,\delta).
\]

\emph{(iv) }For each $x\in D$%
\[
\partial_{\varepsilon}f\left(  x\right)  \cap\left[  \partial_{\varepsilon
}g\left(  x\right)  +B_{\ast}(\theta,K)\right]  \neq\emptyset\text{ \ \ for
all }\varepsilon>0.
\]

\emph{(v) }For each $x\in D$ there exists $\delta>0$ such that%
\[
\partial_{\varepsilon}f\left(  x\right)  \cap\left[  \partial_{\varepsilon
}g\left(  x\right)  +B_{\ast}(\theta,K)\right]  \neq\emptyset\text{ \ \ for
all }\varepsilon\in(0,\delta).
\]

\emph{(vi) }For each $x\in D$%
\[
d\left(  \partial_{\varepsilon}f\left(  x\right)  ,\partial_{\varepsilon
}g\left(  x\right)  \right)  \leq K\text{ \ \ for all }\varepsilon>0.
\]

\emph{(vii) }For each $x\in D$ there exists $\delta>0$ such that%
\[
d\left(  \partial_{\varepsilon}f\left(  x\right)  ,\partial_{\varepsilon
}g\left(  x\right)  \right)  \leq K\text{ \ \ for all }\varepsilon\in
(0,\delta).
\]

\end{theorem}

\begin{proof}
The proofs of the equivalences (i) $\Longleftrightarrow$ (ii)
$\Longleftrightarrow$ (iii) $\Longleftrightarrow$ (iv) $\Longleftrightarrow$
(v) follow from Theorem \ref{dc Lipschitz} by observing that $\partial
_{\varepsilon}(K\left\Vert \cdot\right\Vert )(\theta)=B_{\ast}(\theta,K).$ The
implications (iv) $\Longrightarrow$ (vi) $\Longrightarrow$ (vii) are obvious.
To prove (vii) $\Longrightarrow$ (i), given $x\in D$ we notice that (vii)
implies the existence of $\delta>0$ such that, for all $\gamma>0,$
\[
\partial_{\varepsilon}f\left(  x\right)  \cap\left[  \partial_{\varepsilon
}g\left(  x\right)  +B_{\ast}(\theta,K+\gamma)\right]  \neq\emptyset
\ \ \ \ \text{for all }\varepsilon\in(0,\delta).
\]
Hence, by the equivalence between (v) and (i), $f$ and $g$ are convex, lsc on
$D,$ and $f_{D}-g_{D}$ is Lipschitz with constant $K+\gamma.$ Therefore, since
$\gamma$ is arbitrary, $f_{D}-g_{D}$ is Lipschitz with constant $K.$
\end{proof}

\bigskip

Observing that statements (i), (iv), (v), (vi) and (vii) in Theorem
\ref{dc Lipschitz copy} are symmetric in $f$ and $g,$ it turns out that, under
the assumptions of this theorem, statements (ii) and (iii) are also symmetric;
therefore, if one has%
\[
\emptyset\neq\partial_{\varepsilon}f\left(  x\right)  \subset\partial
_{\varepsilon}g\left(  x\right)  +B^{\ast}\left(  \theta,K\right)  \text{
\ \ for all }\varepsilon>0
\]
for each $x\in D,$ then one also has%
\[
\emptyset\neq\partial_{\varepsilon}g\left(  x\right)  \subset\partial
_{\varepsilon}f\left(  x\right)  +B^{\ast}\left(  \theta,K\right)  \text{
\ \ for all }\varepsilon>0
\]
for each $x\in D.$ We thus obtain the following corollary:

\begin{corollary}
Let $K\geq0.$ If some (hence all) of the statements \emph{(i)--(vii)} of
Theorem \emph{\ref{dc Lipschitz copy}} holds, then for every $x\in D$ and
every $\varepsilon>0$ the Hausdorff distance between $\partial_{\varepsilon
}f\left(  x\right)  $ and $\partial_{\varepsilon}g\left(  x\right)  $ does not
exceed the constant $K.$
\end{corollary}

\begin{corollary}
\label{dc constant normed}The following statements are equivalent\emph{:}

\emph{(i)} $f$ and $g$ are convex, lsc on $D,$ and $f_{D}-g_{D}$ is constant.

\emph{(ii) }For each $x\in D$%
\[
d\left(  \partial_{\varepsilon}f\left(  x\right)  ,\partial_{\varepsilon
}g\left(  x\right)  \right)  =0\text{ \ \ for all }\varepsilon>0.
\]

\emph{(iii) }For each $x\in D$ there exists $\delta>0$ such that
\[
d\left(  \partial_{\varepsilon}f\left(  x\right)  ,\partial_{\varepsilon
}g\left(  x\right)  \right)  =0\text{ \ \ for all }\varepsilon\in(0,\delta).
\]

\end{corollary}

From the previous result we obtain a complement to Corollary \ref{subdiff}:

\begin{corollary}
The following statements are equivalent\emph{:}

\emph{(i)} For each $x\in D$%
\[
\emptyset\neq\partial f\left(  x\right)  =\partial g\left(  x\right)  .
\]

\emph{(ii)} For each $x\in D$%
\[
d\left(  \partial f\left(  x\right)  ,\partial g\left(  x\right)  \right)
=0.
\]

\end{corollary}

\end{document}